\documentclass[preprint,12pt]{elsarticle}

\usepackage[T1]{fontenc}
\usepackage[utf8]{inputenc}
\usepackage{amsmath,amssymb,amsthm,mathtools,mathrsfs}
\usepackage{enumitem}
\usepackage{microtype}
\usepackage{tikz}
\usepackage[hidelinks]{hyperref}
\usepackage{url}

\allowdisplaybreaks
\numberwithin{equation}{section}

\newtheorem{theorem}{Theorem}[section]
\newtheorem{proposition}[theorem]{Proposition}
\newtheorem{lemma}[theorem]{Lemma}
\newtheorem{corollary}[theorem]{Corollary}
\newtheorem{conjecture}[theorem]{Conjecture}

\theoremstyle{remark}
\newtheorem{remark}[theorem]{Remark}

\newcommand{\Z}{\mathbb Z}
\newcommand{\R}{\mathbb R}

\newcommand{\cF}{\mathcal F}
\newcommand{\cC}{\mathcal C}
\newcommand{\cX}{\mathcal X}

\newcommand{\conv}{\operatorname{conv}}
\newcommand{\vol}{\operatorname{vol}}
\newcommand{\diam}{\operatorname{diam}}
\newcommand{\da}{d_{\mathrm a}}
\newcommand{\dAn}{d}
\newcommand{\deltaL}{\delta_{\mathrm L}}

\newcommand{\meet}{\mathbin{\wedge}}

\journal{European Journal of Combinatorics}

\begin{document}

\begin{frontmatter}

\title{An isodiametric theorem and lattice diameter-perfect codes in \texorpdfstring{$A_3$}{A3}}

\author[aff1]{Mladen Kova\v{c}evi\'c\corref{cor1}}
\ead{kmladen@uns.ac.rs}
\cortext[cor1]{Corresponding author.}
\affiliation[aff1]{organization={
University of Novi Sad},
  country={Serbia}}

\begin{abstract}
The root lattice $A_n$, equipped with its graph distance (equivalently, one half of the ambient $\ell_1$ metric), is isometric to $\Z^n$ with the asymmetric Manhattan metric. We study two extremal problems in this space -- the isodiametric problem, i.e., determining the maximum anticode cardinality, and the (non)existence of linear diameter-perfect codes, i.e., lattice tilings by optimal anticodes -- and solve them in dimension $3$. We show that, for every integer $D\ge 0$, the largest cardinality of a diameter-$D$ subset of $A_3$ is
\(
 \binom{D+3}{3}+(D+1)\lfloor D^2/4\rfloor,
\)
and this value is attained by the balanced difference of two discrete simplices.
We then prove an integrality-refined simplex-packing obstruction: a sublattice of $\Z^n$ of asymmetric Manhattan distance greater than $D$ induces a lattice packing by $(D+1)\Delta_n$ in $\R^n$. Combining this observation with the exact lattice-packing density of the tetrahedron yields a complete classification in dimension $3$: lattice diameter-perfect codes in $A_3$ exist precisely for $D=1$ and $D=2$. We also give the equivalent statement for perfect $B_h$ sets of cardinality four. Finally, we formulate a conjecture regarding optimal anticodes in arbitrary dimension, and restate it as an intersection problem for uniform multisets.
\end{abstract}


\end{frontmatter}


\section{Introduction}

An \emph{anticode} in a metric space is a finite set whose pairwise distances are bounded above. The corresponding isodiametric problem asks for the maximum cardinality of a set of prescribed diameter. Such questions are classical in extremal combinatorics and discrete geometry. In the ordinary Manhattan lattice $(\Z^n,\ell_1)$, the problem was studied by Kleitman--Fellows \cite{KleitmanFellows1989}, Du--Kleitman \cite{DuKleitman1990}, Ahlswede--Cai--Zhang \cite{AhlswedeCaiZhang1992}, and Bollob\'as--Leader \cite{BollobasLeader1993}. The present paper treats a type-$A$ analogue whose extremals have a different geometry and where the approaches from the mentioned papers do not seem to be directly applicable.

For $x=(x_1,\ldots,x_n)$ and $y=(y_1,\ldots,y_n)$ in $\Z^n$, define the asymmetric Manhattan metric by
\begin{equation}\label{eq:def-da}
 \da(x,y)=\max\left\{
      \sum_{i=1}^n(x_i-y_i)^+, \;
      \sum_{i=1}^n(y_i-x_i)^+
      \right\},
\end{equation}
where $u^+=\max\{u,0\}$. The function $\da$ is, of course, symmetric in $x$ and $y$;
the adjective \emph{asymmetric} refers to the two one-sided masses in the definition and comes from the fact that this metric arises in the analysis of error-correcting codes for so-called asymmetric channels \cite{Klove1981,Varshamov1973}. The space $(\Z^n,\da)$ is isometric to the root lattice
\begin{equation}
 A_n=\left\{z=(z_0,\ldots,z_n)\in\Z^{n+1} \colon \sum_{i=0}^n z_i=0\right\}
\end{equation}
with graph distance
\begin{equation}
\label{eq:Anmetric}
 \dAn(z,w)=\frac12\sum_{i=0}^n|z_i-w_i|.
\end{equation}
Packings in $(\Z^n, \da)$ and $(A_n, \dAn )$ are geometric problems underlying constructions of error-correcting codes in several different scenarios, including permutation channels \cite{KovacevicTan2018,KovacevicVukobratovic2015} and bit-shift/adjacent transposition channels \cite{Kovacevic2019}. We refer the interested reader to these works and the references therein for more details.


For nonnegative integers $P,Q$, put
\begin{equation}\label{eq:def-SPQ}
 S_n(P,Q)=\left\{x\in\Z^n \colon
     p(x)\coloneq\sum_{i=1}^n x_i^+\le P,
     \quad
     q(x)\coloneq\sum_{i=1}^n(-x_i)^+\le Q
     \right\}.
\end{equation}
Equivalently,
\begin{equation}
 S_n(P,Q)=(P\Delta_n-Q\Delta_n)\cap\Z^n,
 \qquad
 \Delta_n=\conv\!\big\{0,e^{(1)},\ldots,e^{(n)}\big\}.
\end{equation}
These sets have diameter $P+Q$ and cardinality \cite{KovacevicTan2018}
\begin{equation}
\label{eq:general-count-intro}
 |S_n(P,Q)|
 =\sum_{j=0}^n\binom nj\binom Pj
                  \binom{Q+n-j}{n-j}.
\end{equation}
Write
\begin{equation}
 P_D=\left\lceil\frac D2\right\rceil,
 \qquad
 Q_D=\left\lfloor\frac D2\right\rfloor,
 \qquad
 S_{n,D}=S_n(P_D,Q_D).
\end{equation}
Our first main result solves the isodiametric problem in dimension three.

\begin{theorem}[Maximum anticodes in $A_3$]\label{thm:main-anticode}
Let $D\ge0$ be an integer. Every subset $\cF\subseteq\Z^3$ of $\da$-diameter at most $D$ satisfies
\begin{equation}\label{eq:M3-main}
 |\cF|\le M_3(D) \coloneq
 \binom{D+3}{3}+(D+1)\left\lfloor\frac{D^2}{4}\right\rfloor.
\end{equation}
Equality is attained by $S_{3,D}$. Equivalently, the same statement holds in $(A_3,\dAn)$.
\end{theorem}

For even $D=2r$, the extremal construction $S_{3,D}=S_3(r,r)$ is the radius-$r$ ball. For odd $D=2r+1$,
\begin{equation}
 S_{3,D}=S_3(r+1,r)=S_3(r,r) + \big\{0,e^{(1)},e^{(2)},e^{(3)}\big\}.
\end{equation}

A full-rank sublattice $L\le\Z^n$ of minimum distance $D+1$ is called \emph{diameter-perfect} when its index equals the maximum anticode size; this terminology belongs to the general code--anticode framework \cite{AhlswedeAydinianKhachatrian2001}. Our second main result is a complete classification in dimension $3$ for all diameters.

\begin{theorem}[Lattice diameter-perfect codes in $A_3$]
\label{thm:main-tiling}
Let $D\ge1$ be an integer. The following are equivalent:
\begin{enumerate}[label=\textnormal{(\roman*)},leftmargin=2.4em]
 \item $(A_3,\dAn)$ admits a lattice diameter-perfect code of minimum distance $D+1$;
 \item $S_{3,D}$ lattice-tiles $\Z^3$;
 \item $D\in\{1,2\}$.
\end{enumerate}
\end{theorem}

The negative part follows from an integrality refinement of the familiar connection between discrete simplex packings and continuous simplex packings \cite{KovacevicTan2017}: if $L\le\Z^n$ has $\da(L)>D$, then the translates of $(D+1)\Delta_n$ by $L$ have disjoint interiors. Hence, a lattice tiling by $S_{n,D}$ would imply
\begin{equation}\label{eq:obstruction-intro}
 \frac{(D+1)^n}{n!\,|S_{n,D}|}\le\deltaL(\Delta_n),
\end{equation}
where $\deltaL(\Delta_n)$ denotes the lattice-packing density of the standard simplex. In dimension three, Hoylman's exact value $\deltaL(\Delta_3)=18/49$ \cite{Hoylman1970} makes \eqref{eq:obstruction-intro} strong enough to exclude every $D\ge3$.

Sublattices of $\Z^n$ that induce packings of $S_n(P,Q)$ are equivalent to another well-studied type of combinatorial objects -- $B_h$ sets (or Sidon sets of order $h$), with $h=P+Q$ \cite{KovacevicTan2018}. A $B_h$ set is called perfect when this packing is also a tiling \cite{KovacevicTan2017}. Theorem~\ref{thm:main-tiling} can thus be reformulated as follows: perfect $B_h$ sets of cardinality four exist only for $h=1,2$. More generally, the same simplex argument mentioned in the previous paragraph gives a finite lower bound on the order of an Abelian group containing a $B_h$ set.


Finally, we formulate the natural conjecture that $S_{n,D}$ is an optimal anticode for every $n$ and $D$. For fixed $n$ and $D$, the conjecture is equivalent to a fixed-ground-set, fixed-defect Erd\H{o}s--Ko--Rado problem for uniform multisets: determine the largest family of $K$-multisets on an $(n+1)$-element ground set whose pairwise intersections have cardinality at least $K-D$, for sufficiently large $K$. The multiset intersection literature establishes strong results in other parameter ranges \cite{FurediGerbnerVizer2015,LIAO2024,MeagherPurdy2011,MeagherPurdy2016}; the fixed-ground-set regime needed here is not covered by those theorems. Theorem~\ref{thm:main-anticode} settles it when the ground-set size is four.

The paper is organized as follows. Section~\ref{sec:prelim} describes the $A_n$ model and the candidate anticodes. Sections~\ref{sec:seven-slab} and \ref{sec:anticode} prove Theorem~\ref{thm:main-anticode} by the seven-slab reduction, $K_4$ polarization, and terminal enumeration. Section~\ref{sec:tiling} develops the lattice-tiling and simplex-packing consequences and proves Theorem~\ref{thm:main-tiling}. Section~\ref{sec:Bh} gives the $B_h$ set formulation. Section~\ref{sec:conjecture} states the general conjecture and its multiset-intersection equivalent. \ref{sec:appendix} records a direct layer-counting formula for the normalized slab systems.

\section{Preliminaries}
\label{sec:prelim}

\subsection{The \texorpdfstring{$A_n$}{An} model and the candidate anticodes}

The following two lemmas appear in \cite{KovacevicTan2018}; the proofs are given here as well for completeness.

\begin{lemma}[Isometry]
\label{lem:isometry}
The map
\begin{equation}
 \Phi \colon \Z^n\longrightarrow A_n,
 \qquad
 \Phi(x_1,\ldots,x_n)=
 \left(x_1,\ldots,x_n,-\sum_{i=1}^n x_i\right)
\end{equation}
is an isometry from $(\Z^n,\da)$ to $(A_n,\dAn)$.
\end{lemma}

\begin{proof}
Let $u=x-y$, and put $p=p(u)$ and $q=q(u)$.  The last coordinate of $\Phi(x)-\Phi(y)$ equals $q-p$.  Hence
\begin{equation}
 \|\Phi(x)-\Phi(y)\|_1=p+q+|p-q|
 =2\max\{p,q\}=2\da(x,y).
\end{equation}
The map is clearly bijective.
\end{proof}

\begin{lemma}
\label{lem:SPQ-diameter-count}
For $P,Q\ge0$, the set $S_n(P,Q)$ has diameter $P+Q$ when $P+Q>0$, and its cardinality is given by \eqref{eq:general-count-intro}.
\end{lemma}

\begin{proof}
For $x,y\in S_n(P,Q)$,
\(
 p(x-y)\le p(x)+q(y)\le P+Q,
\)
and similarly $q(x-y)\le P+Q$.  Equality is attained by $Pe^{(1)}$ and $-Qe^{(1)}$ if $P,Q>0$; if one parameter is zero, two vertices of the corresponding discrete simplex are at distance $P+Q$.

To count the points, choose the $j$ coordinates on which $x$ is strictly positive.  Their positive integer values, with total at most $P$, can be chosen in $\binom Pj$ ways.  On the remaining $n-j$ coordinates, the quantities $(-x_i)^+$ are nonnegative and have total at most $Q$, giving $\binom{Q+n-j}{n-j}$ choices.
\end{proof}

For dimension three, a convenient symmetric form follows immediately from \eqref{eq:general-count-intro}.

\begin{lemma}\label{lem:S3-count}
If $P+Q=D$, then
\begin{equation}\label{eq:S3-count}
 |S_3(P,Q)|=\binom{D+3}{3}+(D+1)PQ.
\end{equation}
Consequently, among pairs $P,Q\ge0$ with $P+Q=D$, the cardinality is maximized by $\{P,Q\}=\{P_D,Q_D\}$, and the maximum is $M_3(D)$ in \eqref{eq:M3-main}.
\end{lemma}

\begin{proof}
Substituting $n=3$ in \eqref{eq:general-count-intro} and simplifying gives \eqref{eq:S3-count}.  For fixed $D$, the product $PQ$ is maximized by a balanced pair.
\end{proof}

\subsection{Subset sums in dimension three}

For $x\in\Z^3$ and $S\subseteq[3]\coloneq\{1,2,3\}$, write
\begin{equation}
 x(S)=\sum_{i\in S}x_i.
\end{equation}

\begin{lemma}[Subset-sum representation]
\label{lem:subset-sum-metric}
For all $x,y\in\Z^3$,
\begin{equation}
\label{eq:subset-sum-metric}
 \da(x,y)=
 \max_{\varnothing\ne S\subseteq[3]}
 |x(S)-y(S)|.
\end{equation}
\end{lemma}

\begin{proof}
For $z=x-y$,
\begin{equation}
 \max_{S\subseteq[3]}z(S)=\sum_i z_i^+=p(z),
 \qquad
 -\min_{S\subseteq[3]}z(S)=\sum_i(-z_i)^+=q(z).
\end{equation}
The positive support attains the first maximum whenever $p(z)>0$, and the negative support attains the second.  Taking the larger of the two quantities gives \eqref{eq:subset-sum-metric}.  The case $z=0$ is immediate.
\end{proof}

The same description gives a useful characterization of the candidate set.

\begin{lemma}\label{lem:subset-sum-SPQ}
For $P,Q\ge0$,
\begin{equation}\label{eq:subset-sum-SPQ}
 S_3(P,Q)=
 \left\{x\in\Z^3 \colon -Q\le x(S)\le P
       \text{ for every }\varnothing\ne S\subseteq[3]\right\}.
\end{equation}
\end{lemma}

\begin{proof}
If $p(x)\le P$ and $q(x)\le Q$, every subset sum is at most $p(x)$ and at least $-q(x)$.  Conversely, apply the upper inequality to the positive support of $x$ and the lower inequality to its negative support.
\end{proof}

\section{A seven-slab reduction and a discrete polarization}
\label{sec:seven-slab}

For integers $u\le v$, let $[u,v]_{\Z}=\{u,u+1,\ldots,v\}$, and interpret this interval as empty if $u>v$.

\subsection{The slab hull}

\begin{proposition}[Seven-slab reduction]\label{prop:slab-hull}
Let $\cF\subseteq\Z^3$ have diameter at most $D$.  Then there are integer intervals $I_S$ of $D+1$ consecutive integers, one for each nonempty $S\subseteq[3]$, such that
\begin{equation}
 \cF\subseteq X(\mathbf I) \coloneq
 \{x\in\Z^3 \colon x(S)\in I_S\text{ for every }\varnothing\ne S\subseteq[3]\}.
\end{equation}
Moreover, $X(\mathbf I)$ also has diameter at most $D$.
\end{proposition}

\begin{proof}
By Lemma~\ref{lem:subset-sum-metric}, for every nonempty $S$ the integer set $\{x(S) \colon x\in\cF\}$ has width at most $D$, and is therefore contained in an interval of $D+1$ consecutive integers.  The same lemma shows that any two points of the resulting intersection are at distance at most $D$.
\end{proof}

Translate the three coordinates independently so that the singleton intervals are all $[0,D]_{\Z}$.  The remaining intervals have the form
\begin{equation}
 I_{12}=[a,a+D]_{\Z},\quad
 I_{13}=[b,b+D]_{\Z},\quad
 I_{23}=[c,c+D]_{\Z},\quad
 I_{123}=[h,h+D]_{\Z}.
\end{equation}
Define
\begin{align}
 \cX_D(a,b,c,h)=\{(x_1,x_2,x_3)\in[0,D]_{\Z}^3 \colon &
 a\le x_1+x_2\le a+D,\notag\\
 &b\le x_1+x_3\le b+D,\notag\\
 &c\le x_2+x_3\le c+D,\notag\\
 &h\le x_1+x_2+x_3\le h+D\},
 \label{eq:def-XD}
\end{align}
and let
\begin{equation}
 N_D(a,b,c,h)=|\cX_D(a,b,c,h)|.
\end{equation}
Since the total sum of a point in $[0,D]^3$ lies in $[0,3D]$, replacing $h<0$ by $0$, or $h>2D$ by $2D$, can only enlarge the set.  Hence we may assume throughout that
\begin{equation}\label{eq:h-range}
 0\le h\le2D.
\end{equation}

\subsection{The \texorpdfstring{$K_4$}{K4} edge labels}

Associate six integers with the edges of $K_4$ by
\begin{equation}\label{eq:edge-labels}
 \lambda_{12}=a,\quad \lambda_{13}=b,\quad \lambda_{23}=c,
 \qquad
 \lambda_{14}=h-c,\quad \lambda_{24}=h-b,\quad
 \lambda_{34}=h-a.
\end{equation}
Thus opposite edges satisfy
\begin{equation}\label{eq:opposite-sum}
 \lambda_{ij}+\lambda_{k\ell}=h
 \quad\text{whenever }\{i,j,k,\ell\}=[4].
\end{equation}

\begin{figure}[t]
\centering
\begin{tikzpicture}[scale=1.03, every node/.style={font=\small}]
  \coordinate (v1) at (0,2.4);
  \coordinate (v2) at (-2.2,-1.1);
  \coordinate (v3) at (2.2,-1.1);
  \coordinate (v4) at (0,-0.15);
  \draw (v1)--(v2)--(v3)--cycle;
  \draw (v4)--(v1) (v4)--(v2) (v4)--(v3);
  \fill (v1) circle (1.5pt) node[above] {$1$};
  \fill (v2) circle (1.5pt) node[below left] {$2$};
  \fill (v3) circle (1.5pt) node[below right] {$3$};
  \fill (v4) circle (1.5pt) node[below] {$4$};
  \node[left] at (-1.15,0.72) {$a$};
  \node[right] at (1.15,0.72) {$b$};
  \node[below] at (0,-1.08) {$c$};
  \node[right] at (0.18,1.18) {$h-c$};
  \node[above] at (-1.02,-0.58) {$h-b$};
  \node[above] at (1.02,-0.58) {$h-a$};
\end{tikzpicture}
\caption{The six normalized pair-sum offsets as edge labels of $K_4$.  Opposite edge labels sum to $h$.}
\label{fig:K4-labels}
\end{figure}
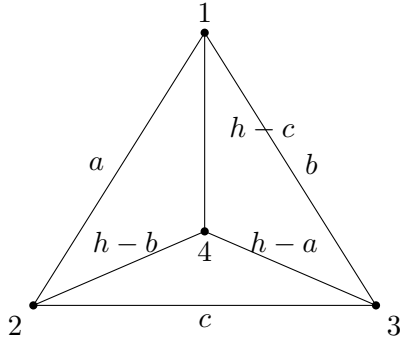

The labeling is displayed in Fig.~\ref{fig:K4-labels}.  Its usefulness comes from the following affine $A_3$ symmetry.

\begin{lemma}[Coordinate charts]\label{lem:K4-charts}
Put
\begin{equation}
 w=(w_1,w_2,w_3,w_4)
 =(x_1,x_2,x_3,-x_1-x_2-x_3)\in A_3.
\end{equation}
Delete any one coordinate $w_\ell$, use the other three as coordinates, and translate their singleton intervals to $[0,D]_{\Z}$.  In the resulting normalized representation, the three pair-sum lower endpoints are precisely the labels on the triangle $K_4-\ell$, and the total-sum lower endpoint is $h$.

Consequently, every permutation of the vertices of $K_4$ induces an affine lattice bijection between normalized slab systems and permutes the six edge labels.
\end{lemma}

\begin{proof}
For $\ell=4$, this is the original chart with pair labels $(\lambda_{12},\lambda_{13},\lambda_{23})=(a,b,c)$.  As an illustration, delete $w_3$.  The singleton intervals for $(w_1,w_2,w_4)$ are $[0,D]$, $[0,D]$, and $[-h-D,-h]$.  Replacing $w_4$ by $w_4+h+D$ normalizes the last interval.  The pair-sum lower endpoints become
\(
 a, h-c, h-b,
\)
which are $\lambda_{12},\lambda_{14},\lambda_{24}$, and the total-sum interval is again $[h,h+D]$.  The other charts are identical by symmetry.  Permuting the four $A_3$ coordinates and then choosing a chart gives the final assertion.
\end{proof}

\subsection{A local polarization}

We first record a one-dimensional overlap fact.

\begin{lemma}\label{lem:interval-shift}
Let $J=[A,B]_{\Z}$ and $K=[L,U]_{\Z}$ be integer intervals.  If
\begin{equation}
 L+U\le A+B-1,
\end{equation}
then
\begin{equation}
 |J\cap(K+1)|\ge |J\cap K|,
 \qquad K+1=[L+1,U+1]_{\Z}.
\end{equation}
\end{lemma}

\begin{proof}
Shifting $K$ one step to the right can only remove the point $L$ and add the point $U+1$.  If $L\notin J$, no point of $J\cap K$ is lost.  If $L\in J$, then
\begin{equation}
 U+1\le A+B-L\le B
\end{equation}
and $U+1\ge L+1>A$, so the new point also belongs to $J$.
\end{proof}

\begin{lemma}[Local polarization]\label{lem:local-polarization}
If $\gamma-\alpha\ge1$, then
\begin{equation}\label{eq:local-polarization}
 N_D(\alpha,\beta,\gamma,h)
 \le N_D(\alpha+1,\beta,\gamma-1,h).
\end{equation}
\end{lemma}

\begin{proof}
Fix $y=x_2$ and $r=x_1+x_3$.  The constraints involving $x_1+x_3$ and $x_1+x_2+x_3$ depend only on $(y,r)$ and remain unchanged.  Put $x=x_1$, so $x_3=r-x$.  The coordinate bounds give
\begin{equation}
 x\in J_0=[\max\{0,r-D\},\min\{D,r\}]_{\Z}.
\end{equation}
The pair constraints with lower endpoints $\alpha$ and $\gamma$ give
\begin{equation}
 x\in J_\alpha=[\alpha-y,\alpha+D-y]_{\Z},
\end{equation}
and
\begin{equation}
 x\in J_\gamma=[y+r-\gamma-D,y+r-\gamma]_{\Z}.
\end{equation}
If $K=J_\alpha\cap J_\gamma$ is nonempty, its midpoint is
\(
 \frac{r+\alpha-\gamma}{2},
\)
whereas the midpoint of $J_0$ is $r/2$.  The hypothesis implies that the midpoint of $K$ lies at least one half-unit to the left of the midpoint of $J_0$, or equivalently that the endpoint sums satisfy the hypothesis of Lemma~\ref{lem:interval-shift}.

Under $(\alpha,\gamma)\mapsto(\alpha+1,\gamma-1)$, both $J_\alpha$ and $J_\gamma$ shift one step to the right, so $K$ is replaced by $K+1$.  Lemma~\ref{lem:interval-shift} shows that the number of admissible values of $x$ in each fixed $(y,r)$ fiber does not decrease.  Summing over all fibers proves \eqref{eq:local-polarization}.  If $K$ is empty, the assertion is immediate.
\end{proof}

In the edge-label formulation, Lemma~\ref{lem:local-polarization} can be applied to any two nonopposite edges, since they lie in a common triangle of $K_4$.  If $e$ and $f$ are nonopposite and
\(
 \lambda_f-\lambda_e\ge2,
\)
we increase $\lambda_e$ by one, decrease $\lambda_f$ by one, and make the opposite changes on the two opposite edges so that \eqref{eq:opposite-sum} remains valid.  The cardinality does not decrease.

\begin{proposition}[Termination of polarizations]\label{prop:termination}
Starting from any normalized slab system satisfying \eqref{eq:h-range}, repeated local polarizations produce, without decreasing cardinality, a label system satisfying
\begin{equation}\label{eq:terminal-condition}
 |\lambda_e-\lambda_f|\le1
 \qquad\text{for every pair of nonopposite edges }e,f.
\end{equation}
\end{proposition}

\begin{proof}
Consider
\begin{equation}
 \Psi(\lambda)=\sum_{e\in E(K_4)}
 \left(\lambda_e-\frac h2\right)^2.
\end{equation}
If two nonopposite labels $p<q-1$ are balanced to $p+1,q-1$, their opposite labels undergo the opposite changes.  The change in $\Psi$ is
\(
 4(p-q+1)<0.
\)
Thus every nontrivial polarization strictly decreases a nonnegative discrete potential, and the process terminates.  At termination no nonopposite labels differ by two or more, which is \eqref{eq:terminal-condition}.
\end{proof}

\section{Terminal systems and proof of the anticode theorem}
\label{sec:anticode}

\subsection{Classification of terminal labels}

\begin{lemma}\label{lem:terminal-classification}
Let an integer edge labeling of $K_4$ satisfy \eqref{eq:opposite-sum} and \eqref{eq:terminal-condition}.
\begin{enumerate}[label=\textnormal{(\alph*)},leftmargin=2.3em]
 \item If $h=2t+1$, every edge label is $t$ or $t+1$, each opposite pair contains one of each, and the three edges labeled $t$ form either a triangle or a star.
 \item If $h=2t$, every edge label belongs to $\{t-1,t,t+1\}$.  Each opposite pair is either $(t,t)$ or $(t-1,t+1)$, and at most one opposite pair is of the latter type.
\end{enumerate}
\end{lemma}

\begin{proof}
Suppose first that $h=2t+1$.  If a label were at most $t-1$, its opposite would be at least $t+2$.  Any third edge is nonopposite to both and could not be within one of each, contradicting \eqref{eq:terminal-condition}.  The case of a label at least $t+2$ is symmetric.  Hence all labels are $t$ or $t+1$, and \eqref{eq:opposite-sum} puts one of each in every opposite pair.  Choosing one edge from each of the three opposite pairs gives either a triangle or a star.

Now let $h=2t$.  The same argument excludes labels outside $\{t-1,t,t+1\}$.  Equation \eqref{eq:opposite-sum} gives the stated possibilities for opposite pairs.  If two opposite pairs were imbalanced, a $t-1$ edge from one and a $t+1$ edge from the other would be nonopposite and differ by two.
\end{proof}

Because of \eqref{eq:h-range}, the integer $t$ in the lemma always satisfies $0\le t\le D$.  Put
\begin{equation}\label{eq:PQ-t}
 P=D-t,\qquad Q=t,
\end{equation}
and define
\begin{equation}
\label{eq:FD}
 F_D(t) \coloneq |S_3(D-t,t)|
 =\binom{D+3}{3}+(D+1)t(D-t).
\end{equation}

\subsection{Counting the terminal systems}

\begin{proposition}\label{prop:terminal-counts}
Up to the affine $S_4$ symmetry of Lemma~\ref{lem:K4-charts}, the terminal systems have the following cardinalities.
\begin{enumerate}[label=\textnormal{(\roman*)},leftmargin=2.5em]
 \item If $h=2t$ and all labels equal $t$, the cardinality is $F_D(t)$.
 \item If $h=2t$ and one opposite pair is labeled $t-1,t+1$, the cardinality is $F_D(t)-(D+1)$.
 \item If $h=2t+1$, both the triangle and the star type have cardinality
 \begin{equation}\label{eq:GD}
  G_D(t) \coloneq F_D(t)-\binom{t+2}{2}+\binom{D-t}{2}.
 \end{equation}
\end{enumerate}
\end{proposition}

\begin{proof}
Let $I=[-Q,P]_{\Z}$, with $P,Q$ as in \eqref{eq:PQ-t}.

\smallskip
\noindent\emph{Balanced even type.}
Take $a=b=c=t$ and $h=2t$.  Under the translation $u_i=x_i-t$, every nonempty subset sum $u(S)$ lies in $I$.  Lemma~\ref{lem:subset-sum-SPQ} identifies the set with $S_3(P,Q)$, giving $F_D(t)$.

\smallskip
\noindent\emph{Imbalanced even type.}
After relabeling, take
$ a=t+1 $, $ b=c=t $, $h=2t$.
After the same translation, all subset-sum intervals are $I$ except the interval for $u_1+u_2$, which is $I+1$.  Relative to $S_3(P,Q)$, the lost points satisfy $u_1+u_2=-Q$.  Since the old set is $S_3(P,Q)$, the equality $u_1+u_2=-Q$ exhausts the entire negative budget.  Hence $u_1,u_2\le0$, with $Q+1$ ordered choices, while $u_3$ must be nonnegative and can be any of $0,1,\ldots,P$.  Thus $(Q+1)(P+1)$ points are lost.

The gained points satisfy $u_1+u_2=P+1$.  The singleton upper bounds force $u_1,u_2\ge1$, giving $P$ ordered choices; the unchanged total-sum interval then forces $u_3\in\{-Q,\ldots,-1\}$, giving $Q$ choices.  Hence $PQ$ points are gained.  The net loss is
\(
 (Q+1)(P+1)-PQ=P+Q+1=D+1
\).

\smallskip
\noindent\emph{Odd triangle type.}
Take $a=b=c=t$ and $h=2t+1$.  After translation, all proper nonempty subset sums lie in $I$, while the total sum lies in $I+1$.  The removed points have total sum $-Q$.  Since they lie in $S_3(P,Q)$, this equality forces positive mass zero and negative mass $Q$; all three coordinates are nonpositive, and their negative masses form a composition of $Q$ into three parts.  Hence their number is $\binom{Q+2}{2}$.  The gained points have total sum $P+1$.  If one coordinate were nonpositive, the sum of the other two would exceed the unchanged pair-sum upper bound $P$.  Thus all three coordinates are positive, giving $\binom P2$ points.  This yields \eqref{eq:GD}.

\smallskip
\noindent\emph{Odd star type.}
After relabeling, take
$ a=b=t $, $ c=t+1 $, $ h=2t+1 $.
After translation, the intervals for
$ C=u_2+u_3 $, $ T=u_1+u_2+u_3 $
are $I+1$, while the other five subset-sum intervals are $I$.  Passing from $S_3(P,Q)$ to this system removes the union of the lower-boundary sets $\{C=-Q\}$ and $\{T=-Q\}$ and adds the union of the corresponding upper-boundary sets $\{C=P+1\}$ and $\{T=P+1\}$.  The relevant counts are
\[
\begin{array}{c|c@{\qquad}c|c}
\text{lower boundary}&\text{cardinality}&\text{upper boundary}&\text{cardinality}\\ \hline
C=-Q&(P+1)(Q+1)&C=P+1&P(Q+1)\\
T=-Q&\binom{Q+2}{2}&T=P+1&\binom{P+1}{2}\\
C=T=-Q&Q+1&C=T=P+1&P.
\end{array}
\]
For example, on $C=-Q$ the variables $u_2,u_3$ are nonpositive with total negative mass $Q$, while $u_1\in\{0,\ldots,P\}$.  On $C=P+1$, the variables $u_2,u_3$ are positive with sum $P+1$, while $u_1\in\{-Q,\ldots,0\}$.  The two total-sum rows follow similarly: $T=-Q$ forces all coordinates nonpositive, whereas $T=P+1$ forces $u_2,u_3\ge1$ and $u_1\ge0$.  The intersection rows are obtained by setting $u_1=0$.
Inclusion--exclusion shows that the number removed minus the number added is
\(
 \binom{Q+2}{2}-\binom P2
\).
The star type therefore has the same cardinality $G_D(t)$ as the triangle type.
\end{proof}

\subsection{Optimization}

\begin{lemma}
\label{lem:optimize-FG}
For $0\le t\le D-1$,
\begin{equation}\label{eq:F-difference}
 F_D(t+1)-F_D(t)=(D+1)(D-2t-1)
\end{equation}
and
\begin{equation}
\label{eq:G-average}
 2G_D(t)=F_D(t)+F_D(t+1)-(D+1).
\end{equation}
Consequently,
\begin{equation}
 G_D(t)<\max\{F_D(t),F_D(t+1)\}.
\end{equation}
\end{lemma}

\begin{proof}
Both identities follow by direct substitution from \eqref{eq:FD} and \eqref{eq:GD}.  The strict inequality follows because $D+1>0$.
\end{proof}

\begin{proof}[Proof of Theorem~\ref{thm:main-anticode}]
By Proposition~\ref{prop:slab-hull}, it suffices to bound a normalized slab system $\cX_D(a,b,c,h)$.  We may assume \eqref{eq:h-range}.  Proposition~\ref{prop:termination} transforms the system, without decreasing its cardinality, into a terminal label system.  Lemma~\ref{lem:terminal-classification} and Proposition~\ref{prop:terminal-counts} show that its cardinality is one of
\begin{equation}
 F_D(t),\qquad F_D(t)-(D+1),\qquad G_D(t).
\end{equation}
The second is smaller than the first, and Lemma~\ref{lem:optimize-FG} bounds the third by a neighboring value of $F_D$.  Finally, \eqref{eq:F-difference} shows that $F_D(t)$ is maximized at
\begin{equation}
 t=Q_D=\left\lfloor\frac D2\right\rfloor
\quad\text{or}\quad
 t=P_D=\left\lceil\frac D2\right\rceil.
\end{equation}
The maximum equals
\begin{equation}
 \binom{D+3}{3}+(D+1)P_DQ_D
 =\binom{D+3}{3}+(D+1)\left\lfloor\frac{D^2}{4}\right\rfloor.
\end{equation}
Lemma~\ref{lem:SPQ-diameter-count} shows that $S_{3,D}$ has diameter $D$ and attains this value.
\end{proof}

Note that the proof determines the optimal cardinality and exhibits extremal families, but it does not classify all equality cases.

\section{Lattice tilings and a simplex-packing obstruction}
\label{sec:tiling}

\subsection{Packings by the balanced anticodes}

Let $L\le\Z^n$ be a full-rank sublattice, and put
\begin{equation}
 \da(L)=\min\{\da(0,\ell) \colon 0\ne\ell\in L\}.
\end{equation}

\begin{lemma}
\label{lem:difference-body-discrete}
If $P+Q=D$, then
\begin{equation}
\label{eq:difference-body-discrete}
 S_n(P,Q)-S_n(P,Q)
 =\{z\in\Z^n \colon \da(z,0)\le D\}.
\end{equation}
Consequently, $S_n(P,Q)+L$ is a lattice packing if and only if $\da(L)>D$, and it is a lattice tiling if and only if, in addition,
\begin{equation}
 [\Z^n:L]=|S_n(P,Q)|.
\end{equation}
\end{lemma}

\begin{proof}
One inclusion in \eqref{eq:difference-body-discrete} follows from the diameter bound.  For the converse, let $z\in\Z^n$ satisfy
\begin{equation}
 A \coloneq p(z)\le D,\qquad B \coloneq q(z)\le D.
\end{equation}
On the positive support of $z$, choose integers $0\le \alpha_i\le z_i$ with
\begin{equation}
 \sum_{z_i>0}\alpha_i=A_1 \coloneq \max\{0,A-Q\}.
\end{equation}
Such a choice exists because $A_1\le A$.  Put $\beta_i=z_i-\alpha_i$ there.  Then
$ \sum\alpha_i=A_1\le P $, $ \sum\beta_i=A-A_1\le Q $.
On the negative support choose integers $0\le\gamma_i\le -z_i$ with
\begin{equation}
 \sum_{z_i<0}\gamma_i=B_1 \coloneq \max\{0,B-P\},
\end{equation}
and put $\eta_i=-z_i-\gamma_i$.  Thus
$ \sum\gamma_i=B_1\le Q $, $ \sum\eta_i=B-B_1\le P $.
Define
\begin{equation}
 x_i=\begin{cases}
      \alpha_i,&z_i>0,\\
      -\gamma_i,&z_i<0,\\
      0,&z_i=0,
     \end{cases}
 \qquad
 y_i=\begin{cases}
      -\beta_i,&z_i>0,\\
      \eta_i,&z_i<0,\\
      0,&z_i=0.
     \end{cases}
\end{equation}
Then $x,y\in S_n(P,Q)$ and $z=x-y$.

Two lattice translates meet precisely when their translation-vector difference belongs to \eqref{eq:difference-body-discrete}.  In a packing, the quotient map $\Z^n\to\Z^n/L$ is injective on $S_n(P,Q)$.  It is bijective, and hence gives a tiling, exactly when the tile size equals the quotient order.
\end{proof}

In dimension three, Theorem~\ref{thm:main-anticode} and Lemma~\ref{lem:difference-body-discrete} give the promised equivalence.  We use the standard diameter-perfect terminology from the code--anticode literature; see, for example, \cite{AhlswedeAydinianKhachatrian2001}.

\begin{corollary}
\label{cor:diam-perfect-tiling}
A full-rank lattice $L\le\Z^3$ is a diameter-perfect code of minimum distance $D+1$ if and only if $S_{3,D}+L$ is a lattice tiling.
\end{corollary}

\begin{proof}
If $\da(L)=D+1$ and $[\Z^3:L]=M_3(D)$, Lemma~\ref{lem:difference-body-discrete} gives a packing of density one, hence a tiling.  Conversely, a tiling gives $[\Z^3:L]=M_3(D)$ and $\da(L)>D$.  If $\da(L)\ge D+2$, then $S_{3,D+1}+L$ would be a packing, contradicting
\begin{equation}
 |S_{3,D+1}|=M_3(D+1)>M_3(D)=[\Z^3:L].
\end{equation}
Thus $\da(L)=D+1$.
\end{proof}

\subsection{The integrality-refined simplex bound}

\begin{theorem}[Simplex-packing obstruction]\label{thm:simplex-obstruction}
Let $L\le\Z^n$ be a full-rank lattice with $\da(L)>D$.  Then
\begin{equation}
 ((D+1)\Delta_n,L)
\end{equation}
is a lattice packing in $\R^n$.  Consequently, if $S_{n,D}+L$ is a lattice tiling, then
\begin{equation}
\label{eq:simplex-obstruction}
 \frac{(D+1)^n}{n!\,|S_{n,D}|}\le\deltaL(\Delta_n).
\end{equation}
\end{theorem}

\begin{proof}
Suppose that the interiors of $(D+1)\Delta_n$ and $(D+1)\Delta_n+\ell$ overlap for some $0\ne\ell\in L$.  Then $\ell=u-v$ for points $u,v$ in the interior of $(D+1)\Delta_n$.  Since $u_i,v_i>0$ and $\sum_i u_i,\sum_i v_i<D+1$,
\begin{equation}
 p(\ell)=\sum_{\ell_i>0}(u_i-v_i)
 \le\sum_i u_i<D+1,
\end{equation}
and similarly $q(\ell)<D+1$.  Both are integers, so $p(\ell),q(\ell)\le D$, contrary to $\da(L)>D$.  Hence the simplex translates have disjoint interiors.

If $S_{n,D}+L$ tiles, then $\det L=[\Z^n:L]=|S_{n,D}|$.  The induced simplex packing has density
\begin{equation}
 \frac{\vol((D+1)\Delta_n)}{\det L}
 =\frac{(D+1)^n}{n!\,|S_{n,D}|},
\end{equation}
which cannot exceed $\deltaL(\Delta_n)$.
\end{proof}

The replacement of $D$ by $D+1$ is a finite, integrality-based strengthening of the discrete-to-continuous simplex-packing connection used in \cite{KovacevicTan2017}.

\subsection{Dimension $n=3$}

For later use, write the cardinalities according to parity:
\begin{equation}\label{eq:M3-even}
 M_3(2r)=\frac{(2r+1)(5r^2+5r+3)}{3},
\end{equation}
\begin{equation}\label{eq:M3-odd}
 M_3(2r-1)=\frac{2r(5r^2+1)}{3}.
\end{equation}
Hoylman proved
\begin{equation}\label{eq:Hoylman}
 \deltaL(\Delta_3)=\frac{18}{49}.
\end{equation}

\begin{proposition}\label{prop:no-tiling-Dge3}
For every $D\ge3$, the set $S_{3,D}$ does not lattice-tile $\Z^3$.
\end{proposition}

\begin{proof}
First let $D=2r-1$ with $r\ge2$.  If a lattice tiling existed, Theorem~\ref{thm:simplex-obstruction}, \eqref{eq:M3-odd}, and \eqref{eq:Hoylman} would imply
\begin{equation}
 \frac{(2r)^3}{6M_3(2r-1)}
 =\frac{2r^2}{5r^2+1}
 \le\frac{18}{49}.
\end{equation}
But the reverse strict inequality is equivalent to $8r^2>18$, which holds for $r\ge2$.

Now let $D=2r$ with $r\ge2$.  Equations \eqref{eq:M3-even} and \eqref{eq:simplex-obstruction} would give
\begin{equation}
 \frac{(2r+1)^2}{2(5r^2+5r+3)}\le\frac{18}{49}.
\end{equation}
The reverse strict inequality is equivalent to
\(
 16r^2+16r-59>0
\),
which holds for every $r\ge2$.
\end{proof}

\begin{proof}[Proof of Theorem~\ref{thm:main-tiling}]
The known constructions give lattice tilings for $D=1$ in every dimension and for $D=2$ in every prime-power dimension; see \cite[Theorems 9 and 11]{KovacevicTan2018}.  In particular, both exist in dimension three.
Proposition~\ref{prop:no-tiling-Dge3} excludes all $D\ge3$, and Corollary~\ref{cor:diam-perfect-tiling} gives the coding-theoretic formulation.
\end{proof}

The diameter-perfect lattices for $D=1,2$ can be given explicitly:
\begin{align}
L_1 &= \{ x \colon x_1 + 2x_2 + 3x_3 = 0 \pmod 4\} , \\
L_2 &= \{ x \colon x_1 + 3x_2 + 9x_3 = 0 \pmod{13}\} .
\end{align}

\section{Consequences for \texorpdfstring{$B_h$}{Bh} sets}
\label{sec:Bh}

Let $G$ be a finite Abelian group, written additively.  A set
\begin{equation}
 B=\{0,b_1,\ldots,b_n\}\subseteq G
\end{equation}
is a $B_h$ set if all sums
\begin{equation}
 \alpha_1b_1+\cdots+\alpha_nb_n,
 \qquad
 \alpha_i\in\Z_{\ge0},\quad \sum_i\alpha_i\le h,
\end{equation}
are distinct.

The standard group-splitting correspondence associates with $B$ the lattice
\begin{equation}
\label{eq:Bh-lattice}
 L_B=\left\{x\in\Z^n \colon \sum_{i=1}^n x_ib_i=0\text{ in }G\right\}.
\end{equation}
If $B$ generates $G$, then
\begin{equation}\label{eq:Bh-correspondence}
 \da(L_B)>h,
 \qquad
 \Z^n/L_B\cong G,
 \qquad
 \det L_B=|G|.
\end{equation}
Conversely, every full-rank lattice of asymmetric Manhattan distance greater than $h$ gives a generating $B_h$ set in its quotient group.  This correspondence and its applications to $A_n$ codes and simplex packings are discussed in \cite{Kovacevic2022,KovacevicTan2017,KovacevicTan2018}.

Let $\phi(h,n)$ be the smallest order of an Abelian group containing a $B_h$ set of cardinality $n+1$.  In minimizing the group order, one may assume that the set generates the ambient group, since otherwise the group may be replaced by the subgroup generated by the set.

\begin{corollary}[A finite simplex bound for $B_h$ sets]\label{cor:phi-bound}
For all $h,n\ge1$,
\begin{equation}
\label{eq:phi-bound}
 \phi(h,n)\ge
 \left\lceil
 \frac{(h+1)^n}{n!\,\deltaL(\Delta_n)}
 \right\rceil.
\end{equation}
In particular,
\begin{equation}
 \phi(h,3)\ge
 \left\lceil\frac{49(h+1)^3}{108}\right\rceil.
\end{equation}
\end{corollary}

\begin{proof}
Apply Theorem~\ref{thm:simplex-obstruction} to the lattice in \eqref{eq:Bh-lattice} and use \eqref{eq:Bh-correspondence}.
\end{proof}

Following the terminology suggested in \cite{KovacevicTan2017}, call a generating $B_h$ set of cardinality $n+1$ \emph{perfect} if the corresponding lattice tiles $\Z^n$ by $S_{n,h}$; equivalently,
\(
 |G|=|S_{n,h}|
\).

\begin{corollary}[Perfect $B_h$ sets of cardinality four]\label{cor:perfect-Bh-four}
A perfect $B_h$ set of cardinality four exists if and only if $h\in\{1,2\}$.
\end{corollary}

\begin{proof}
The group-splitting correspondence identifies perfect $B_h$ sets of cardinality four with lattice tilings of $\Z^3$ by $S_{3,h}$.  Apply Theorem~\ref{thm:main-tiling}.  For $h=1$, one may take all elements of a group of order four; for $h=2$, Singer's planar difference set of order three gives a set of cardinality four in a group of order thirteen \cite{Singer1938}.
\end{proof}

\begin{remark}
The bound \eqref{eq:phi-bound} replaces the factor $h^n$ in the elementary continuous-simplex lower bound by $(h+1)^n$.  It is exact for every $h$ when $n=1$ or $n=2$, using $\deltaL(\Delta_1)=1$, $\deltaL(\Delta_2)=2/3$, and the known perfect constructions in these dimensions \cite{KovacevicTan2017,KovacevicTan2018}.
\end{remark}

\section{A conjecture in arbitrary dimension and intersecting multisets}
\label{sec:conjecture}

Define
\begin{equation}
 A_n(D)=\max\{|\cF| \colon \cF\subseteq\Z^n,\ \diam_{\da}(\cF)\le D\}.
\end{equation}
Theorem~\ref{thm:main-anticode} says $A_3(D)=|S_{3,D}|$.  The natural generalization is the following.

\begin{conjecture}[Optimal anticodes in $A_n$]
\label{conj:general-anticode}
For all integers $n,D\ge0$,
\begin{equation}
\label{eq:general-anticode-conj}
 A_n(D)=M_n(D) \coloneq |S_{n,D}|
 =\sum_{j=0}^n\binom nj\binom{P_D}{j}
       \binom{Q_D+n-j}{n-j}.
\end{equation}
\end{conjecture}

We next give an exact reformulation as a multiset intersection problem.  For integers $m\ge2$ and $K\ge0$, let
\begin{equation}
 \Omega_{m,K}=\left\{a=(a_1,\ldots,a_m)\in\Z_{\ge0}^m \colon \sum_{i=1}^m a_i=K\right\}.
\end{equation}
This is the space of $K$-multisets on an $m$-element ground set.  For $a,b\in\Omega_{m,K}$, define
\begin{equation}
 |a\meet b|=\sum_{i=1}^m\min(a_i,b_i),
\end{equation}
the cardinality of their multiset intersection.  Then
\begin{equation}\label{eq:multiset-metric}
 |a\meet b|
 =\frac12\sum_{i=1}^m(a_i+b_i-|a_i-b_i|)
 =K-\frac12\|a-b\|_1.
\end{equation}
Thus a family $\cF\subseteq\Omega_{m,K}$ is $(K-D)$-intersecting if and only if it has half-$\ell_1$ diameter at most $D$ in the affine copy of $A_{m-1}$.

\begin{proposition}[Fixed-defect multiset reformulation]\label{prop:multiset-equivalence}
For fixed $m\ge2$ and $D\ge0$, Conjecture~\ref{conj:general-anticode} in dimension $m-1$ is equivalent to the following statement: for all sufficiently large $K$, every $(K-D)$-intersecting family $\cF\subseteq\Omega_{m,K}$ satisfies
\begin{equation}
\label{eq:fixed-defect-multiset}
 |\cF|\le M_{m-1}(D).
\end{equation}
The bound is attained for all sufficiently large $K$.
\end{proposition}

\begin{proof}
The layer $\Omega_{m,K}$ is contained in an affine translate of $A_{m-1}$, and \eqref{eq:multiset-metric} identifies its metric with half the ambient $\ell_1$ distance.  Hence Conjecture~\ref{conj:general-anticode} immediately gives \eqref{eq:fixed-defect-multiset}.

Conversely, let $\cF\subset A_{m-1}$ be finite of diameter at most $D$.  For $N$ sufficiently large, translation by $(N,\ldots,N)$ embeds $\cF$ into $\Omega_{m,mN}$ without changing distances.  Applying \eqref{eq:fixed-defect-multiset} gives the conjectured anticode bound.  Finally, translating the image of $S_{m-1,D}$ sufficiently far into the nonnegative orthant shows attainability for all sufficiently large $K$-layers.
\end{proof}

The natural complete-intersection candidates are also visible in this formulation.  Let $0\le s\le D$, put
\begin{equation}
 P_s=D-s,\qquad Q_s=s,
\end{equation}
and choose $T\in\Z_{\ge0}^m$ with
\begin{equation}
 |T|=K-D+2s.
\end{equation}
Define
\begin{equation}
\label{eq:multiset-C_s}
 \cC_s(T)=\left\{a\in\Omega_{m,K} \colon |a\meet T|\ge K-D+s\right\}.
\end{equation}
The coordinatewise inequality
\begin{equation}
 |a\meet b|\ge |a\meet T|+|b\meet T|-|T|
\end{equation}
shows that $\cC_s(T)$ is $(K-D)$-intersecting.

\begin{lemma}\label{lem:complete-intersection-size}
If
\begin{equation}
\label{eq:T-interior}
 Q_s\le T_i\le K-P_s
 \qquad(1\le i\le m),
\end{equation}
then
\begin{equation}
 |\cC_s(T)|=|S_{m-1}(P_s,Q_s)|.
\end{equation}
\end{lemma}

\begin{proof}
For $a\in\cC_s(T)$ put $z=a-T$.  Since $|a|=K$ and $|T|=K-D+2s$,
\begin{equation}
 \sum_i z_i=P_s-Q_s.
\end{equation}
Moreover,
\begin{equation}
 u\coloneq\sum_i z_i^+=K-|a\meet T|\le P_s,
 \qquad
 v\coloneq\sum_i(-z_i)^+\le Q_s.
\end{equation}
Delete the final coordinate of $z$, and let $x\in\Z^{m-1}$ be the resulting vector.  Write $p=p(x)$, $q=q(x)$, and $r=\sum_{i<m}x_i=p-q$.  The missing coordinate is $z_m=P_s-Q_s-r$, and hence
\begin{equation}
 u=\max\{p,q+P_s-Q_s\},
 \qquad
 v=\max\{q,p-P_s+Q_s\}.
\end{equation}
It follows that $u\le P_s$ and $v\le Q_s$ are equivalent to
\begin{equation}
 p(x)\le P_s,\qquad q(x)\le Q_s.
\end{equation}
This gives an injection into $S_{m-1}(P_s,Q_s)$.  Conversely, every $x$ in that set reconstructs a vector $z$ with each coordinate in $[-Q_s,P_s]$ and with the required total sum; condition \eqref{eq:T-interior} guarantees $a=T+z\in\Omega_{m,K}$.  The correspondence is therefore bijective.
\end{proof}

For large $K$, interior choices of $T$ exist.  Lemma~\ref{lem:complete-intersection-size} shows that the balanced choice $s=Q_D$ gives the conjectured value $M_{m-1}(D)$.  Thus Conjecture~\ref{conj:general-anticode} is a fixed-ground-set, fixed-defect complete-intersection assertion for multisets.

For ordinary sets, the corresponding complete-intersection problem is solved by the theorem of Ahlswede--Khachatrian \cite{AhlswedeKhachatrian1997}.  The multiset Erd\H{o}s--Ko--Rado literature uses precisely the multiplicity-sensitive intersection in \eqref{eq:multiset-metric}.  Meagher and Purdy determined the ordinary intersecting case and its equality cases \cite{MeagherPurdy2011}.  F\"uredi, Gerbner, and Vizer used down-compression to determine the maximum cardinality of a $t$-intersecting family of $K$-multisets in the large-ground-set range $m\ge 2K-t$ \cite{FurediGerbnerVizer2015}; Meagher and Purdy subsequently determined the equality cases in that range and established related Hilton--Milner-type results \cite{MeagherPurdy2016}.  Here
\begin{equation}
 m\text{ is fixed},\qquad t=K-D,\qquad
 2K-t=K+D\longrightarrow\infty,
\end{equation}
so the available large-ground-set theorem does not apply.  Theorem~\ref{thm:main-anticode} settles this fixed-defect problem for $m=4$.

\section{Concluding remarks}
\label{sec:conclusion}

Theorem~\ref{thm:main-tiling} concerns lattice tilings.  It does not rule out arbitrary nonlattice translational tilings by $S_{3,D}$.  For every fixed dimension at least three, general tilings are known not to exist for all sufficiently large diameters by a compactness argument together with the necessary central symmetry of facets of translative tiling polytopes \cite{McMullen1980,KovacevicTan2018}.  A complete nonlattice classification in dimension three remains open, as does the problem of obtaining effective general-tiling thresholds in higher dimensions.

The central open problem suggested by this work is Conjecture~\ref{conj:general-anticode}.  The $K_4$ polarization used here is special to rank three: the six complementary pair-sum offsets form three opposite edge pairs, and local balancing leaves only four terminal types.  In higher rank, the analogous subset-sum parameters live on a much larger Boolean lattice.  The multiset formulation may offer a more natural setting for general compression or complete-intersection methods.

\appendix

\section{A direct layer-counting formula}\label{app:layer-counting}
\label{sec:appendix}

For completeness, we record an exact formula for the normalized slab count $N_D(a,b,c,h)$.  It gives an alternative, piecewise-polynomial description of the four-parameter problem and independently reproduces the terminal counts in Proposition~\ref{prop:terminal-counts}.

Fix the total sum
\begin{equation}
 s=x_1+x_2+x_3.
\end{equation}
The pair constraints in \eqref{eq:def-XD} become independent coordinate bounds
\begin{equation}
\begin{aligned}
 \ell_1(s)&=\max\{0,s-c-D\}, \;& u_1(s)&=\min\{D,s-c\},\\
 \ell_2(s)&=\max\{0,s-b-D\},& u_2(s)&=\min\{D,s-b\},\\
 \ell_3(s)&=\max\{0,s-a-D\},& u_3(s)&=\min\{D,s-a\}.
\end{aligned}
\end{equation}
Put
\begin{equation}
 m_i(s)=u_i(s)-\ell_i(s)+1,
 \qquad
 R(s)=s-\sum_{i=1}^3\ell_i(s).
\end{equation}
For an integer $r$, define
\begin{equation}
 \binom r2_+=
 \begin{cases}
  \binom r2,&r\ge2,\\
  0,&r<2.
 \end{cases}
\end{equation}
If $u_i(s)<\ell_i(s)$ for some $i$, set $T_s(a,b,c)=0$.  Otherwise, inclusion--exclusion for bounded compositions gives the number of points in the $s$-layer as
\begin{equation}\label{eq:layer-formula}
 T_s(a,b,c)=
 \sum_{J\subseteq[3]}(-1)^{|J|}
 \binom{R(s)-\sum_{j\in J}m_j(s)+2}{2}_+.
\end{equation}
Consequently,
\begin{equation}\label{eq:N-layer-sum}
 N_D(a,b,c,h)=
 \sum_{s=\max\{0,h\}}^{\min\{3D,h+D\}}
 T_s(a,b,c).
\end{equation}

Substitution of the four terminal label systems gives, respectively,
\begin{equation}
 F_D(t),\qquad F_D(t)-(D+1),\qquad G_D(t),\qquad G_D(t).
\end{equation}
Moreover, comparing \eqref{eq:layer-formula} before and after the local move $(\alpha,\gamma)\mapsto(\alpha+1,\gamma-1)$ reduces fiberwise to the interval-overlap inequality of Lemma~\ref{lem:interval-shift}.  Thus the direct enumeration and the polarization proof are two forms of the same local phenomenon.

\section*{Funding}
This work was supported by the Ministry of Science, Technological Development and Innovation of the Republic of Serbia [contract no. 451-03-34/2026-03/200156] and by the Faculty of Technical Sciences, University of Novi Sad, Serbia [project no. 01-3609/1]. The funders had no role in the design of the study, the analysis or interpretation of the results, the preparation of the manuscript, or the decision to submit it for publication.

\sloppy
\hfuzz=1pt
\bibliographystyle{ejc-num}
\bibliography{A3_optimal_anticodes}

@article{AhlswedeCaiZhang1992,
  author  = {Ahlswede, Rudolf and Cai, Ning and Zhang, Zhen},
  title   = {Diametric theorems in sequence spaces},
  journal = {Combinatorica},
  volume  = {12},
  number  = {1},
  pages   = {1--17},
  year    = {1992},
  doi     = {10.1007/BF01191200}
}

@article{AhlswedeKhachatrian1997,
  author  = {Ahlswede, Rudolf and Khachatrian, Levon H.},
  title   = {The complete intersection theorem for systems of finite sets},
  journal = {European Journal of Combinatorics},
  volume  = {18},
  number  = {2},
  pages   = {125--136},
  year    = {1997},
  doi     = {10.1006/eujc.1995.0092}
}

@article{AhlswedeAydinianKhachatrian2001,
  author  = {Ahlswede, Rudolf and Aydinian, Harout K. and Khachatrian, Levon H.},
  title   = {On perfect codes and related concepts},
  journal = {Designs, Codes and Cryptography},
  volume  = {22},
  number  = {3},
  pages   = {221--237},
  year    = {2001},
  doi     = {10.1023/A:1008394205999}
}

@article{BollobasLeader1993,
  author  = {Bollob{\'a}s, B{\'e}la and Leader, Imre},
  title   = {Maximal sets of given diameter in the grid and the torus},
  journal = {Discrete Mathematics},
  volume  = {122},
  number  = {1--3},
  pages   = {15--35},
  year    = {1993},
  doi     = {10.1016/0012-365X(93)90284-Z}
}

@article{DuKleitman1990,
  author  = {Du, Ding-Zhu and Kleitman, Daniel J.},
  title   = {Diameter and radius in the {M}anhattan metric},
  journal = {Discrete \& Computational Geometry},
  volume  = {5},
  pages   = {351--356},
  year    = {1990},
  doi     = {10.1007/BF02187795}
}

@article{FurediGerbnerVizer2015,
  author  = {F{\"u}redi, Zolt{\'a}n and Gerbner, D{\'a}niel and Vizer, M{\'a}t{\'e}},
  title   = {A discrete isodiametric result: the {E}rd{\H{o}}s--{K}o--{R}ado theorem for multisets},
  journal = {European Journal of Combinatorics},
  volume  = {48},
  pages   = {224--233},
  year    = {2015},
  doi     = {10.1016/j.ejc.2015.02.023}
}

@article{Hoylman1970,
  author  = {Hoylman, Donald J.},
  title   = {The densest lattice packing of tetrahedra},
  journal = {Bulletin of the American Mathematical Society},
  volume  = {76},
  pages   = {135--137},
  year    = {1970},
  doi     = {10.1090/S0002-9904-1970-12400-4}
}

@article{KleitmanFellows1989,
  author  = {Kleitman, Daniel J. and Fellows, Michael R.},
  title   = {Radius and diameter in {M}anhattan lattices},
  journal = {Discrete Mathematics},
  volume  = {73},
  number  = {1--2},
  pages   = {119--125},
  year    = {1989},
  doi     = {10.1016/0012-365X(88)90139-2}
}

@book{Klove1981,
  author  = {Kl{\o}ve, Torleiv},
  title   = {Error correcting codes for the asymmetric channel},
  publisher  = {Dept. Inform., Univ. Bergen, Bergen, Norway, Tech. Rep.},
  year    = {1981}
}

@article{Kovacevic2019,
  author  = {Kova{\v{c}}evi{\'c}, Mladen},
  title   = {Runlength-limited sequences and shift-correcting codes: asymptotic analysis},
  journal = {IEEE Transactions on Information Theory},
  volume  = {65},
  number  = {8},
  pages   = {4804--4814},
  year    = {2019},
  doi     = {10.1109/TIT.2019.2907979}
}

@article{Kovacevic2022,
  author  = {Kova{\v{c}}evi{\'c}, Mladen},
  title   = {Abelian difference sets as lattice coverings and lattice tilings},
  journal = {Bulletin of the Australian Mathematical Society},
  volume  = {106},
  number  = {2},
  pages   = {177--184},
  year    = {2022},
  doi     = {10.1017/S0004972721001271}
}

@article{KovacevicTan2017,
  author  = {Kova{\v{c}}evi{\'c}, Mladen and Tan, Vincent Y. F.},
  title   = {Improved bounds on {S}idon sets via lattice packings of simplices},
  journal = {SIAM Journal on Discrete Mathematics},
  volume  = {31},
  number  = {3},
  pages   = {2269--2278},
  year    = {2017},
  doi     = {10.1137/16M1099182}
}

@article{KovacevicTan2018,
  author  = {Kova{\v{c}}evi{\'c}, Mladen and Tan, Vincent Y. F.},
  title   = {Codes in the space of multisets---coding for permutation channels with impairments},
  journal = {IEEE Transactions on Information Theory},
  volume  = {64},
  number  = {7},
  pages   = {5156--5169},
  year    = {2018},
  doi     = {10.1109/TIT.2017.2789292}
}

@article{KovacevicVukobratovic2015,
  author  = {Kova{\v{c}}evi{\'c}, Mladen and Vukobratovi{\'c}, Dejan},
  title   = {Perfect codes in the discrete simplex},
  journal = {Designs, Codes and Cryptography},
  volume  = {75},
  number  = {1},
  pages   = {81--95},
  year    = {2015},
  doi     = {10.1007/s10623-013-9893-5}
}

@article{LIAO2024,
title = {{E}rdős-{K}o-{R}ado theorem for bounded multisets},
journal = {Journal of Combinatorial Theory, Series A},
volume = {206},
pages = {105888},
year = {2024},
issn = {0097-3165},
doi = {10.1016/j.jcta.2024.105888},
author = {Jiaqi Liao and Zequn Lv and Mengyu Cao and Mei Lu},
}

@article{McMullen1980,
  author  = {McMullen, Peter},
  title   = {Convex bodies which tile space by translation},
  journal = {Mathematika},
  volume  = {27},
  number  = {1},
  pages   = {113--121},
  year    = {1980},
  doi     = {10.1112/S0025579300010007}
}

@article{MeagherPurdy2011,
  author  = {Meagher, Karen and Purdy, Alison},
  title   = {An {E}rd{\H{o}}s--{K}o--{R}ado theorem for multisets},
  journal = {Electronic Journal of Combinatorics},
  volume  = {18},
  number  = {1},
  pages   = {P220},
  year    = {2011},
  doi     = {10.37236/707}
}

@article{MeagherPurdy2016,
  author  = {Meagher, Karen and Purdy, Alison},
  title   = {Intersection theorems for multisets},
  journal = {European Journal of Combinatorics},
  volume  = {52},
  pages   = {120--135},
  year    = {2016},
  doi     = {10.1016/j.ejc.2015.09.006}
}

@article{Singer1938,
  author  = {Singer, James},
  title   = {A theorem in finite projective geometry and some applications to number theory},
  journal = {Transactions of the American Mathematical Society},
  volume  = {43},
  number  = {3},
  pages   = {377--385},
  year    = {1938},
  doi     = {10.1090/S0002-9947-1938-1501951-4}
}

@article{Varshamov1973,
  author  = {Varshamov, Rom R.},
  title   = {A class of codes for asymmetric channels and a problem from the additive theory of numbers},
  journal = {IEEE Transactions on Information Theory},
  volume  = {19},
  number  = {1},
  pages   = {92--95},
  year    = {1973},
  doi     = {10.1109/TIT.1973.1054954}
}

\end{document}